\documentclass[letterpaper,11pt]{amsart}

\usepackage[margin=1.2in]{geometry}
\usepackage{amsmath,amsthm,amssymb}
\usepackage{xspace,xcolor}
\usepackage[breaklinks,colorlinks,citecolor=teal,linkcolor=teal,urlcolor=teal,pagebackref,hyperindex]{hyperref}
\usepackage[alphabetic]{amsrefs}
\usepackage[all]{xy}
\usepackage{color}

\theoremstyle{plain}
\newtheorem{thm}{Theorem}[section]

\newtheorem{lem}[thm]{Lemma}
\newtheorem{prop}[thm]{Proposition}
\newtheorem{cor}[thm]{Corollary}

\theoremstyle{definition}

\newtheorem{eg}[thm]{Example}

\theoremstyle{remark}
\newtheorem{rmk}[thm]{Remark}

\def\Z{{\mathbf Z}}
\def\Q{{\mathbf Q}}

\def\C{{\mathbf C}}

\def\Qalg{\overline{\mathbf Q}}

\def\A{{\mathbf A}}

\def\cJ{\mathcal{J}}

\def\cM{\mathcal{M}}

\def\cO{\mathcal{O}}

\def\J{\mathcal{J}}

\def\.{\cdot}
\def\^{\widehat}

\def\({\left(}
\def\){\right)}

\newcommand{\llbracket}{[\negthinspace[}
\newcommand{\rrbracket}{]\negthinspace]}

\renewcommand{\and}{ \ \ \text{ and } \ \ }

\DeclareMathOperator{\Spec} {Spec}

\DeclareMathOperator{\ord} {ord}

\DeclareMathOperator{\lct} {lct}

\DeclareMathOperator{\moi}{moi}
\DeclareMathOperator{\Vol}{vol}
\def\oi{{_K\rm{oi}}}

\begin{document}

\author[R.~Cluckers]
{Raf Cluckers}
\address{Universit\'e de Lille, Laboratoire Painlev\'e, CNRS - UMR 8524, Cit\'e Scientifique, 59655
Villeneuve d'Ascq Cedex, France, and
KU Leuven, Department of Mathematics,
Celestijnenlaan 200B, B-3001 Leu\-ven, Bel\-gium}
\email{Raf.Cluckers@univ-lille.fr}
\urladdr{http://rcluckers.perso.math.cnrs.fr/}

\author[M.~Musta{\c{t}}{\u{a}}]{Mircea Musta{\c{t}}{\u{a}}}
\address{Department of Mathematics, University of Michigan, 530 Church Street, Ann Arbor, MI 48109, USA}
\email{mmustata@umich.edu}

\thanks{R.C. was partially supported by the European Research Council under the European Community's Seventh Framework Programme (FP7/2007-2013) with ERC Grant Agreement nr. 615722
MOTMELSUM, by the Labex CEMPI  (ANR-11-LABX-0007-01), and by KU Leuven IF C14/17/083. M.M. was partially supported by NSF grant DMS-1701622.}


\begin{abstract}
We show that if $f$ is a nonzero, noninvertible function on a smooth complex variety $X$ and $J_f$ is the Jacobian
ideal of $f$, then $\lct(f,J_f^2)>1$ if and only if the hypersurface defined by $f$ has rational singularities. Moreover,
if this is not the case, then $\lct(f,J_f^2)=\lct(f)$. We give two proofs, one relying on arc spaces and one that
goes through the inequality
$\widetilde{\alpha}_f\geq\lct(f,J_f^2)$, where $\widetilde{\alpha}_f$ is the minimal exponent of $f$.
In the case of a polynomial over $\overline{\Q}$, we also prove an analogue of this latter inequality,
with $\widetilde{\alpha}_f$  replaced by the motivic oscillation index $\moi(f)$.
\end{abstract}

\title{An invariant detecting rational singularities via the log canonical threshold}

\maketitle

\section{Introduction}

Given a smooth complex algebraic variety $X$ and a (nonempty) hypersurface $H$ in $X$ defined by $f\in\cO_X(X)$, the log canonical
threshold $\lct(f)$ measures how far the pair $(X,H)$ is from having log canonical singularities. In particular, we always have
$\lct(f)\leq 1$, with equality if and only if the pair $(X,H)$ is log canonical. The log canonical threshold $\lct({\mathfrak a})$ can be defined
more generally for every nonzero coherent ideal ${\mathfrak a}$ of $\cO_X$, with the convention that $\lct(\cO_X)=\infty$.
For an introduction to singularities of pairs in our setting, we refer the reader to \cite[Chapter~9]{Lazarsfeld}.

The main point of this note is that one can use the log canonical
threshold of an ideal associated to $f$ in order to refine $\lct(f)$, so that we detect when the hypersurface $H$ has rational singularities.
Namely, we consider the Jacobian ideal $J_f$ of $f$ and the log canonical threshold $\lct(f,J_f^2)$ of the ideal $(f)+J_f^2$.
We show that $\lct(f,J_f^2)>1$ if and only if the hypersurface $H$ has rational singularities. More precisely, we have the following:

\begin{thm}\label{thmA_intro}
For every smooth, complex algebraic variety $X$, and every nonzero, noninvertible $f\in \cO_X(X)$ defining the hypersurface
$H$ in $X$, the following hold:
\begin{enumerate}
\item[i)] If $H$ does not have rational singularities, then
$$\lct(f,J_f^2)=\lct(f).$$
In particular, we have $\lct(f,J_f^2)\leq 1$.
\item[ii)] If $H$ has rational singularities, then $\lct(f,J_f^2)>1$.
\end{enumerate}
\end{thm}

Another invariant that refines $\lct(f)$ and detects whether $H$ has rational singularities is the \emph{minimal exponent} $\widetilde{\alpha}_f$,
see \cite{Saito-B}. When $H$ has isolated singularities, this has been also known as the \emph{complex singularity index} of $H$. In general,
it is defined as the negative of the largest root of $b_f(s)/(s+1)$, where $b_f(s)$ is the Bernstein-Sato polynomial of $f$ (with the convention
that if $f$ defines a smooth hypersurface, in which case $b_f(s)=s+1$, then $\widetilde{\alpha}_f=\infty$).
It is a result of Lichtin and
Koll\'{a}r (see \cite[Theorem~1.6]{Kollar}) that $\lct(f)=\min\{\widetilde{\alpha}_f,1\}$ and it was shown by Saito (see \cite[Theorem~0.4]{Saito-B})
that $\widetilde{\alpha}_f>1$ if and only if $H$ has rational singularities. We thus see that $\widetilde{\alpha}_f$ behaves like $\lct(f,J_f^2)$.
We prove the following general inequality between these two invariants:

\begin{thm}\label{thmB_intro}
For every smooth, complex algebraic variety $X$, and every nonzero, noninvertible $f\in \cO_X(X)$, we have
$$\widetilde{\alpha}_f\geq\lct(f,J_f^2).$$
\end{thm}

We prove Theorem~\ref{thmB_intro} by using results on minimal exponents from \cite{MP} and a theorem of Varchenko
saying that in a family of isolated singularities with constant Milnor number, the minimal exponent is constant.
Regarding Theorem~\ref{thmA_intro}, we note that the interesting assertion is the one in i), as the one in ii) follows easily from known
properties of rational singularities. Part i) is a consequence of the result in Theorem~\ref{thmB_intro}; however,
we give a second proof using arc spaces. This has the advantage that it also gives
the assertion below concerning divisorial valuations. A \emph{divisorial valuation} is a valuation of the form $v=q\cdot {\rm ord}_E$,
where $E$ is a divisor on a normal variety that has a birational morphism to $X$ and $q$ is a positive integer; we denote by $A_X(v)$
the log discrepancy of $v$ and by $c_X(v)$ the center of $v$ on $X$ (for definitions, see Section~\ref{first_proof}).

\begin{thm}\label{thm3_intro}
Let $X$ be a smooth, affine, complex algebraic variety, and $f\in\cO_X(X)$ a nonzero function.
If $v$ is a divisorial valuation on $X$ such that
$$0<v(J_f)<\frac{1}{2}v(f),$$
then there is a divisorial valuation $w$ on $X$ that satisfies the following conditions:
\begin{enumerate}
\item[i)] $w(g)\leq v(g)$ for every $g\in\cO_X(X)$,
\item[ii)] $w(f)\geq v(f)-1$,
\item[iii)] $A_X(w)\leq A_X(v)-1$, and
\item[iv)] $c_X(w)=c_X(v)$.
\end{enumerate}
\end{thm}

As a consequence of this theorem, we obtain the following result
concerning multiplier ideals, which in turn immediately implies Theorem~\ref{thmA_intro}:

\begin{cor}\label{corC_intro}
For every smooth, complex algebraic variety $X$, and every nonzero, noninvertible $f\in \cO_X(X)$, we have
$$\cJ(X,f^{\lambda})=\cJ\big((f,J_f^2)^{\lambda}\big)\quad\text{for all}\quad \lambda<1.$$
Moreover, if $f$ defines a reduced hypersurface, then
$${\rm adj}(f)=\cJ(f,J_f^2),$$
where ${\rm adj}(f)$ is the adjoint ideal of $f$.
\end{cor}

There is another invariant that behaves like $\widetilde{\alpha}_f$ and $\lct(f,J_f^2)$, namely the \emph{motivic oscillation index}
$\moi(f)$ studied in \cite{CMN}. This is defined for polynomials $f\in \overline{\Q}[x_1,\ldots,x_n]$ and
it was shown in
\cite[Proposition~3.10]{CMN} that $\lct(f)=\min\{\moi(f),1\}$ and $\moi(f)>1$ if and only if the hypersurface defined by $f$ in $\A_{\overline{\Q}}^n$
has rational singularities.

In fact, a more refined version $\moi_Z(f)$ of the motivic oscilation index also involves a closed subscheme $Z$ of $\A^n_{\overline{\Q}}$
(the one we referred to in the previous paragraph corresponds to the case when $Z$ is the hypersurface defined by $f$).
We recall the precise definition of $\moi_Z(f)$ in Section~\ref{moi}. We only mention now that if $f\in\Z[x_1,\ldots,x_n]$ and
$Z=\A^n_{\overline{\Q}}$, then $\moi_Z(f)$ relates to finite exponential sums over integers modulo $p^m$ (for primes $p$ and integers $m>0$)
of the form
$$
E(p^m) := \frac{1}{p^{mn}}\sum_{x\in(\Z/p^m\Z)^n} \exp \left(2\pi i \frac{f(x)}{p^m}\right)
$$
and to certain limit values of all possible $\sigma\geq 0$ such that
$$
|E(p^m)|  \ll p^{-m\sigma },
$$
with an implicit constant independent from $m$.
These limits are taken carefully, using in fact finite field extensions and large primes $p$, as will be described in Section~\ref{moi}.
We have the following inequality between $\moi(f)$ and $\lct(f,J_f^2)$:

\begin{thm}\label{thm_moi1}
For every nonconstant $f\in{\overline{\Q}}[x_1,\ldots,x_n]$, we have
\begin{equation}\label{eq_thm_moi1}
\moi(f)\geq\lct(f,J_f^2).
\end{equation}
Moreover, if $\lct(f,J_f^2)\leq 1$, then we have equality in (\ref{eq_thm_moi1}).
\end{thm}

Note that by Theorem~\ref{thmA_intro}, we have $\lct(f,J_f^2)\leq 1$ if and only if the hypersurface defined by $f$
does not have rational singularities. In fact, one can give a third proof of part i) of Theorem ~\ref{thmA_intro} by combining 
Theorem~\ref{thm_moi1} and the mentioned assertions in \cite[Proposition~3.10]{CMN}; we leave the details to the reader.

We will prove a more general version of the above theorem, allowing for a subset
 $Z$ of the zero-locus of $f$ (see Theorem~\ref{thm_moi2} below).
A related intriguing question is whether we always have  $\widetilde{\alpha}_f=\moi(f)$. However, investigating this
seems to require new ideas.

The paper is organized as follows. In Section~\ref{first_proof}, we prove Theorem~\ref{thmB_intro} and deduce
Theorem~\ref{thmA_intro}. In Section~\ref{second_proof}, after reviewing some basic facts about the connection between valuations and contact loci
in arc spaces, we prove Theorem~\ref{thm3_intro}, deduce Corollary~\ref{corC_intro}, and obtain a second proof of Theorem~\ref{thmA_intro}.
In Section~\ref{examples}, we give two examples. We show that for generic determinantal hypersurfaces, the inequality in
Theorem~\ref{thmB_intro} is an equality, and we describe when this inequality is strict in the case of homogeneous diagonal hypersurfaces.
Finally, in Section~\ref{moi} we recall the definition of the motivic oscillation index and prove the general version of Theorem~\ref{thm_moi1}.

\subsection{Acknowledgments}
We would like to thank Mattias Jonsson, Johannes Nicaise and Mihai P\u{a}un for useful discussions and to Nero Budur for his comments on an earlier version
of this paper.

\section{The inequality between $\widetilde{\alpha}_f$ and $\lct(f,J_f^2)$}\label{first_proof}

In what follows $X$ is a smooth (irreducible) $n$-dimensional, complex algebraic variety.
For basic facts about log canonical thresholds and multiplier ideals we refer to \cite[Chapter~9]{Lazarsfeld}.
Let us begin by recalling some terminology and notation regarding valuations that will be used both in this section and the next one.

A \emph{divisorial valuation} on $X$
is a valuation of the function field $k(X)$ of $X$ of the form $q\cdot {\rm ord}_E$, where $q$ is a positive integer
and $E$ is a prime divisor on a normal variety $Y$ that has a birational morphism $g\colon Y\to X$
(here ${\rm ord}_E$ is the discrete valuation associated to $E$, with corresponding DVR $\cO_{Y,E}$, having
fraction field $k(Y)=k(X)$). After replacing $Y$ by a suitable log resolution of $(Y,E)$, we may always assume that
$Y$ is smooth and $E$ is a smooth prime divisor on $Y$. The \emph{center} of $q\cdot {\rm ord}_E$ on $X$
is the closure of $g(E)$ (which is independent of the model $Y$). The \emph{log discrepancy} of $q\cdot {\rm ord}_E$
is the positive integer
$$A_X(q\cdot {\rm ord}_E)=q\cdot \big({\rm ord}_E(K_{Y/X})+1\big),$$
where $K_{Y/X}$ is the effective divisor on $Y$ locally defined by the determinant of the Jacobian matrix of $g$.

We first give a proposition
concerning the minimal exponent of a general linear combination of the generators of an ideal.

\begin{prop}\label{gen_combination}
If $f_1,\ldots,f_r\in\cO_X(X)$ generate the proper nonzero coherent ideal ${\mathfrak a}$ of $\cO_X$ and
$f=\sum_{i=1}^r\lambda_if_i$, with $\lambda_1,\ldots,\lambda_r\in\C$ general, then we have
$$\widetilde{\alpha}_f\geq\lct({\mathfrak a}).$$
\end{prop}

\begin{proof}
If the zero-locus $Z$ of ${\mathfrak a}$ has codimension 1 in $X$, then ${\rm lct}({\mathfrak a})\leq 1$ and in this case we have
$$\widetilde{\alpha}_f\geq {\rm lct}(f)={\rm lct}({\mathfrak a}).$$
The equality follows from the fact that for every $t\in (0,1]$, we have the equality of multiplier ideals
$${\mathcal J}(f^t)={\mathcal J}({\mathfrak a}^t)$$
(see \cite[Proposition~9.2.26]{Lazarsfeld}). We thus may and will assume that ${\rm codim}_X(Z)\geq 2$.

The argument then proceeds as in \emph{loc. cit}. Let $\pi\colon Y\to X$ be a log resolution of $(X, {\mathfrak a})$ that is an isomorphism
over $X\smallsetminus Z$. By construction, if we put ${\mathfrak a}\cdot\cO_Y=\cO_Y(-E)$, then $E$ is a simple normal crossing divisor
such that if we write $E=\sum_{i=1}^Na_iE_i$, then every $E_i$ is a $\pi$-exceptional divisor. Since $\lambda_1,\ldots,\lambda_r$ are general,
it follows that if $D$ is the divisor defined by $f$, then $\pi^*(D)=F+E$, where $F$ is a smooth divisor, with no common components with $E$, and
having simple normal crossings with $E$. In particular, $D$ is a reduced divisor and $\pi$ is a log resolution of $(X,D)$ such that the strict transform of $D$ is smooth. We thus deduce
using  \cite[Corollary~D]{MP}
that if $K_{Y/X}=\sum_{i=1}^Nk_iE_i$, then
$$\widetilde{\alpha}_f\geq\min_{i=1}^N\frac{k_i+1}{a_i}=\lct({\mathfrak a}).$$
\end{proof}

In what follows we will also make use of a local version of the minimal exponent. Recall that if $f\in\cO_X(X)$ is nonzero and $P\in X$ is such that
$f(P)=0$, then $\max\{\widetilde{\alpha}_{f\vert_U}\mid U\ni P\}$, where $U$ varies over the open neighborhoods of $P$,
 is achieved for all small enough $U$. This maximum is denoted $\widetilde{\alpha}_{f,P}$ and we have $\widetilde{\alpha}_f=\min
\{ \widetilde{\alpha}_{f,P}\mid P\in X\}$.

\begin{rmk}\label{rmk_gen_combination}
With the same notation as in Proposition~\ref{gen_combination}, for every $P\in X$, if $\lambda_1,\ldots,\lambda_r\in \C$ are general, then we
have
$$\widetilde{\alpha}_{f,P}\geq\lct_P({\mathfrak a}).$$
Indeed, if we choose an open neighborhood $U$ of $P$ such that $\lct_P({\mathfrak a})=\lct({\mathfrak a}\vert_U)$, then by applying the proposition
for ${\mathfrak a}\vert_U$, we obtain
$$\widetilde{\alpha}_{f,P}\geq \widetilde{\alpha}(f\vert_U)\geq \lct({\mathfrak a}\vert_U)=\lct_P({\mathfrak a}).$$
\end{rmk}

Given $f\in\cO_X(X)$, we denote by $J_f$ the \emph{Jacobian ideal} of $f$. Recall that if $U$ is an open subset of $X$
such that $x_1,\ldots,x_n$ are algebraic local coordinates on $U$ (that is, $dx_1,\ldots,dx_n$ give a trivialization of
$\Omega_U$), then $J_f$ is generated on $U$ by $f,\frac{\partial f}{\partial x_1},\ldots,
\frac{\partial f}{\partial x_n}$ (the fact that this only depends on the ideal generated by $f$, but not on the particular generator of the ideal or on the system of coordinates is well-known and straightforward to check). Note that $P\in X$ is an isolated point in the zero-locus of $J_f$ if and only if
$f$ has an isolated singular point at $P$. In this case, we also consider the \emph{Milnor number} $\mu_P(f)$: if $x_1,\ldots,x_n$ is an algebraic
system of coordinates centered at $P$, then $\mu_P(f)=\ell\big(\cO_{X,P}/(\partial f/\partial x_1,\ldots,\partial f/\partial x_n)\big)$. Note that
while the ideal $(\partial f/\partial x_1,\ldots,\partial f/\partial x_n)$ might depend on the system of coordinates, its colength does not.

If $f$ has an isolated singularity at $P\in X$, then Steenbrink \cite{Steenbrink} and Varchenko \cite{Varchenko2} defined a mixed Hodge structure
on the vanishing cohomology of $f$ at $P$, and using this, together with the monodromy action, one defines the \emph{spectrum} ${\rm Sp}_P(f)$,
which is a set of rational numbers, with multiplicities, in the interval $(-1,n-1)$. The sum of the numbers in the spectrum, counted with multiplicities,
is equal to the Milnor number $\mu_P(f)$. It was shown by Malgrange \cite{Malgrange} that
$$\widetilde{\alpha}_{f,P}=1+\min\{\beta\mid\beta\in {\rm Sp}_P(f)\}.$$

We will make use of the following result of Varchenko \cite{Varchenko} about the behavior of the spectrum in families. Suppose that we have a smooth morphism
$\pi\colon Y\to T$, a section $s\colon T\to Y$ of $\pi$, and $g\in\cO_Y(Y)$ such that for every $t\in T$, the restriction $g_t$ of $g$ to the fiber
$Y_t=\pi^{-1}(t)$ is nonzero and $g\big(s(t)\big)=0$. If
$T$ is connected, and $g_t$ has an isolated singularity at $s(t)$, with $\mu_{s(t)}(g_t)$ independent of $t\in T$, then
the spectrum of $g_t$ at $s(t)$ is independent of $t$; in particular,
$\widetilde{\alpha}_{g_t,s(t)}$ is independent of $t$.

We can now prove the inequality between $\widetilde{\alpha}_f$ and $\lct(f,J_f^2)$:

\begin{proof}[Proof of Theorem~\ref{thmB_intro}]
Since
$$\widetilde{\alpha}_f=\min_{P\in Z}\widetilde{\alpha}_{f,P}\quad\text{and}\quad \lct\big((f)+J_f^2\big)=\min_{P\in Z}\lct_P\big((f)+J_f^2\big),$$
where $Z$ is the hypersurface defined by $f$,
it follows that in order to prove the inequality in the theorem, it is enough to show that for every $P\in Z$, we have
\begin{equation}\label{eq_thm_inequality}
\widetilde{\alpha}_{f,P}\geq\lct_P\big((f)+J_f^2\big).
\end{equation}
We may and will assume that $f$ has a singular point at $P$, since otherwise, by convention, both sides of (\ref{eq_thm_inequality})
are infinite. After replacing $X$ by a suitable affine open neighborhood
of $P$, we may assume that $X$ is affine and we have an algebraic system of coordinates $x_1,\ldots,x_n$ centered at $P$.

We first show that it is enough to treat the case when $f$ has an isolated singular point at $P$.  Given any $N\geq 2$, let
$f_N=f+\sum_{i=1}^n a_{N,i}x_i^N$, with $a_{N,1},\ldots,a_{N,n}\in \C$ general. Since the zero locus of the linear system generated by
$f,x_1^N,\ldots,x_n^N$ consists just of $P$, it follows from Kleiman's version of Bertini's theorem that $f_N$ has an isolated singular point at $P$.
On one hand, we have by \cite[Proposition~6.7]{MP}
$$|\widetilde{\alpha}_{f_N,P}-\widetilde{\alpha}_{f,P}|\leq \frac{n}{N}.$$
On the other hand, the ideals $(f)+J_f^2$ and $(f_N)+J_{f_N}^2$ are equal mod $(x_1,\ldots,x_n)^{N-1}$, hence
$$|\lct_P(f,J_f^2)-\lct_P(f_N,J_{f_N}^2)|\leq\frac{n}{N-1}$$
(see, for example, \cite[Property~1.21]{Mustata}).
We thus deduce that if we know the inequality (\ref{eq_thm_inequality}) for each $f_N$, by letting $N$ go to infinity,
we obtain the same inequality for $f$. From now on, we assume that $f$ has an isolated singular point at $P$.

For every $g\in\cO_X(X)$, we put
$$J'_g=\left(\frac{\partial g}{\partial x_1},\ldots,\frac{\partial g}{\partial x_n}\right).$$
Given $\lambda=(\lambda_{i,j})_{1\leq i,j\leq n}\in\C^{n^2}$, we consider
$$g_{\lambda}:=f+\sum_{i,j=1}^n\lambda_{i,j}\frac{\partial f}{\partial x_i}\frac{\partial f}{\partial x_j}.$$
Since $P$ is a singular point of $f$, we have $g_{\lambda}(P)=0$.
Note that $g_0=f$ and for every $\lambda\in\C^{n^2}$ and every $i$, we have
$$\frac{\partial g_{\lambda}}{\partial x_i}-\frac{\partial f}{\partial x_i}\in J'_f\quad\text{for}\quad 1\leq i\leq n.$$
In particular, we have
$$J'_{g_{\lambda}}\subseteq J'_f\quad\text{for all}\quad\lambda\in\C^{n^2},$$
and thus
\begin{equation}\label{eq2_thm_inequality}
\ell\big(\cO_{X,P}/J'_f)\leq \ell(\cO_{X,P}/J'_{g_{\lambda}})\quad\text{for all}\quad\lambda\in\C^{n^2}.
\end{equation}
By the semicontinuity theorem for fiber dimensions, the set
$$U:=\{\lambda\in\C^{n^2}\mid \ell(\cO_{X,P}/J'_{g_{\lambda}})<\infty\}$$
is open in $\C^{n^2}$. Moreover, by the upper semicontinuity of the Milnor number,
for every nonnegative integer $m$, the set
$$U_m:=\{\lambda\in U\mid \ell(\cO_{X,P}/J'_{g_{\lambda}})\leq m\}$$
is open in $U$.
Let $m_0=\ell(\cO_{X,P}/J'_f)=\mu_P(f)$, so that $0\in U_{m_0}$.
It follows from (\ref{eq2_thm_inequality}) that $\mu_P(g_{\lambda})=\mu_P(f)$ for every $\lambda\in U_{m_0}$, hence by Varchenko's result,
we have
\begin{equation}\label{eq3_thm_inequality}
\widetilde{\alpha}_{g_{\lambda},P}=\widetilde{\alpha}_{f,P}\quad\text{for all}\quad\lambda\in U_{m_0}.
\end{equation}
On the other hand, it follows from Remark~\ref{rmk_gen_combination} that for $\lambda\in U_{m_0}$ general,
we have
$$\widetilde{\alpha}_{g_{\lambda},P}\geq \lct_P\big((f)+{J'}_f^2\big)=\lct_P\big((f)+J_f^2\big).$$
This completes the proof of the theorem.
\end{proof}

We now deduce the fact that $\lct(f,J_f^2)$ detects rational singularities.

\begin{proof}[Proof of Theorem~\ref{thmA_intro}]
It follows from
\cite[Theorem~0.4]{Saito-B} that the hypersurface defined by $f$ does not have rational singularities if and only if
$\widetilde{\alpha}_f\leq 1$, in which case we have $\widetilde{\alpha}_f=\lct(f)$. We deduce from
Theorem~\ref{thmB_intro} that in this case $\lct(f)\geq \lct\big((f)+J_f^2\big)$, while the reverse inequality simply follows from the inclusion
$(f)\subseteq (f)+J_f^2$. This proves i).

The assertion in ii) is straightforward: suppose that $H$ has rational singularities and
let $\pi\colon Y\to X$ be a log resolution of $(X,H)$ that is at the same time a log resolution
of the ideal $(f)+J_f^2$. Note first that if $E$ is a prime divisor on $Y$ such that ${\rm ord}_E(f,J_f^2)>0$, then ${\rm ord}_E(J_f)>0$,
and thus $E$ is a $\pi$-exceptional divisor (since $H$ has rational singularities, it is in particular reduced, hence ${\rm ord}_D(J_f)=0$
for every irreducible component $D$ of $H$). Furthermore,
$H$ has rational singularities if and only if it has canonical singularities by a result of Elkik (see \cite[Theorem~11.1]{Kollar}); moreover,
this is the case if and only if the pair
$(X,H)$ has canonical singularities by a result of Stevens (see \cite[Theorem~7.9]{Kollar}). Since $(X,H)$ has canonical singularities and
$E$ is exceptional, we have
$$A_X({\rm ord}_E)\geq {\rm ord}_E(f)+1\geq {\rm ord}_E(f,J_f^2)+1.$$
This holds for all prime divisors $E$ on $Y$ for which ${\rm ord}_E(f,J_f^2)>0$, hence we conclude that
$\lct(f,J_f^2)>1$.
\end{proof}

\section{An approach to $\lct(f,J_f^2)$ via arcs}\label{second_proof}

In this section we use the approach to valuations via arcs to prove Theorem~\ref{thm3_intro}, which we apply to
deduce Corollary~\ref{corC_intro} and give another proof of Theorem~\ref{thmA_intro}.
We keep the assumption that $X$ is a smooth (irreducible)  complex algebraic variety, of dimension $n$.

We first review briefly the definition of jet schemes and the arc scheme. For details, see for example \cite{EM}.
For every $m\geq 0$, the $m^{\rm th}$ \emph{jet scheme} $X_m$ of $X$ is a scheme over $X$ with the property that for every
$\C$-algebra $A$, we have a functorial bijection
$${\rm Hom}({\rm Spec} A,X_m)\simeq {\rm Hom}\big({\rm Spec} A[t]/(t^{m+1}),X\big).$$
In particular, the points of $X_m$ are in canonical bijection with the \emph{$m$-jets} on $X$, that is,
maps ${\rm Spec}\,\C[t]/(t^{m+1})\to X$.
Given such an $m$-jet $\gamma\colon {\rm Spec}\,\C[t]/(t^{m+1})\to X$,
we denote by $\gamma(0)$ the image of the closed point and by $\gamma^*$ the induced ring homomorphism
$\cO_{X,\gamma(0)}\to\C[t]/(t^{m+1})$.
Truncation induces morphisms $X_m\to X_p$ whenever $p<m$ and these satisfy the obvious compatibilities.
Note that we have a canonical isomorphism  $X_0\simeq X$ and we denote by $\pi_m$ the truncation morphism
$X_m\to X$.
With this notation, we have $\pi_m(\gamma)=\gamma(0)$.

All truncation morphisms are affine, hence we may consider the projective limit $X_{\infty}$ of the system $(X_m)_{m\geq 0}$.
This is the \emph{space of arcs} (or \emph{arc scheme}) of $X$. Its $\C$-valued points are in canonical bijection with maps $\Spec\,\C\llbracket t\rrbracket\to X$.
In what follows, we identify $X_{\infty}$ with the corresponding set of $\C$-valued points.
For an arc $\gamma$, we use the notation $\gamma(0)$ and $\gamma^*$ as above. Note that the space of arcs $X_{\infty}$ comes
endowed with truncation maps $\psi_m\colon X_{\infty}\to X_m$ compatible with the truncation morphisms between jet schemes.

Since $X$ is smooth and $n$-dimensional, every morphism $X_m\to X$ is locally trivial in the Zariski topology, with fiber $\A^{mn}$.
In fact, if $x_1,\ldots,x_n$ are algebraic coordinates on an open subset $U$ of $X$,
then we have an isomorphism
$$\pi_m^{-1}(U)\simeq U\times(\A^m)^n,$$
which maps $\gamma$ to $\big(\gamma(0),\gamma^*(x_1),\ldots,\gamma^*(x_n)\big)$
(note that each $\gamma^*(x_i)$ lies in $t\C[t]/t^{m+1}\C[t]\simeq\C^m$).
In particular, every $X_m$ is a smooth, irreducible variety, of dimension $(m+1)n$.
Moreover, using the above isomorphisms we see that each truncation morphism $X_m\to X_p$,
with $p<m$, is locally trivial, with fiber $\A^{(m-p)n}$.

We next turn to the connection between divisorial valuations and certain subsets in the space of arcs.
A \emph{cylinder} in $X_{\infty}$ is a subset of the form $C=\psi_m^{-1}(S)$, where $S$ is a constructible subset of $X_m$.
In this case $C$ is \emph{irreducible}, \emph{closed}, \emph{open}, or \emph{locally closed}
(with respect to the Zariski topology on $X_{\infty}$)
if and only if $S$ has this property. In particular, we have irreducible decomposition
for locally closed cylinders.
The \emph{codimension}
of a cylinder $C=\psi_m^{-1}(S)$ is defined as
$${\rm codim}(C)={\rm codim}_{X_m}(S).$$
It is clear that this is independent of the way we write $C$ as the inverse image of a constructible set.

An important example of cylinders is provided by \emph{contact loci}. Given a coherent ideal ${\mathfrak a}$ in $\cO_X$,
we put ${\rm ord}_{\gamma}({\mathfrak a})\in\Z_{\geq 0}\cup\{\infty\}$ to be that $m$ such that $\gamma^*({\mathfrak a})\C\llbracket t\rrbracket
=(t^m)$, with the convention that
${\rm ord}_{\gamma}({\mathfrak a})=\infty$ if $\gamma^*({\mathfrak a})=0$. The set ${\rm Cont}^{\geq m}({\mathfrak a})$ of $X_{\infty}$
consisting of all arcs $\gamma$ with ${\rm ord}_{\gamma}({\mathfrak a})\geq m$ is a closed cylinder in $X_{\infty}$. We similarly define the
locally closed cylinder ${\rm Cont}^m({\mathfrak a})$.

It turns out that one can use cylinders in $X_{\infty}$ in order to describe divisorial valuations on $X$. We simply state the results
and refer for details and proofs to \cite{ELM}. We assume, for simplicity, that $X$ is affine, though everything extends to the general case
in a straightforward way. For every closed, irreducible cylinder $C\subseteq X_{\infty}$, that does not dominate
$X$ and every $h\in\cO_X(X)$ nonzero, we put
$${\rm ord}_C(h):=\min\{{\rm ord}_{\gamma}(h)\mid\gamma\in C\}\in\Z_{\geq 0}.$$
This extends to a valuation of the function field of $X$; in fact, this is a divisorial valuation, whose center is the
closure of $\psi_0(C)$.

Conversely, if $v=q\cdot {\rm ord}_E$ for some smooth prime divisor $E$ on the smooth variety $Y$, with a birational
morphism $g\colon Y\to X$,
then we have an induced morphism $g_{\infty}\colon Y_{\infty}\to X_{\infty}$ and if $C(v)$ is the closure of $g_{\infty}\big({\rm Cont}^{\geq q}(\cO_Y(-E))\big)$,
then $C(v)$ is an irreducible closed cylinder in $X_{\infty}$ and ${\rm ord}_{C(v)}=v$. Moreover, a key fact is that
$${\rm codim}\big(C(v)\big)=A_X(v).$$
In addition, for every closed, irreducible cylinder $C\subseteq X_{\infty}$ that does not dominate $X$, if $v={\rm ord}_C$, then
$C\subseteq C(v)$. In particular, we have ${\rm codim}(C)\geq A_X(v)$.

After this overview, we can prove our general result about valuations.

\begin{proof}[Proof of Theorem~\ref{thm3_intro}]
Let $C=C(v)$, so that ${\rm ord}_C=v$ and ${\rm codim}(C)=A_X(v)$.
We put $Z=c_X(v)$, so that $Z$ is the closure of $\psi_0(C)$.
If $m=v(f)$ and $e=v(J_f)$,
we have by hypothesis $0<e<\frac{1}{2}m$. Since $m={\ord}_C(f)$ and $e={\rm ord}_C(J_f)$, we have $C\subseteq {\rm Cont}^{\geq m}(f)\cap
{\rm Cont}^{\geq e}(J_f)$.
Let $C_0:=C\cap {\rm Cont}^e(J_f)$, which is a nonempty subcylinder of $C$, open in $C$.
Since $C$ is irreducible, we have $C=\overline{C_0}$.
We also consider the locally closed cylinder
$$C':={\rm Cont}^{\geq (m-1)}(f)\cap{\rm Cont}^e(J_f)\cap \psi_0^{-1}(Z).$$
It is clear that $C_0$ is a closed subset of $C'$. We make the following

\noindent {\bf Claim}. $C_0$ is \emph{not} an irreducible component of $C'$.

Assuming the claim, let $W$ be an irreducible component of $C'$ that contains $C_0$ and
$\overline{W}$ its closure in $X_{\infty}$. Note that $\overline{W}$ is an irreducible, closed cylinder
in $X_{\infty}$ such that $\psi_0(\overline{W})\subseteq Z$. Since we also have
$$Z\subseteq\overline{\psi_0(C_0)}\subseteq\overline{\psi_0(\overline{W})},$$
we conclude that $\overline{\psi_0(\overline{W})}=Z$. Therefore $w:={\rm ord}_{\overline{W}}$
is a divisorial valuation on $X$, with center $Z$.

Since $W\subseteq {\rm Cont}^{\geq m-1}(f)$, it follows that $w(f)\geq m-1$. Furthermore, since
$C_0\subseteq W$, we have $C=\overline{C_0}\subseteq\overline{W}$, hence we clearly have
$$w(g)={\rm ord}_{\overline{W}}(g)\geq {\rm ord}_C(g)=v(g)\quad\text{for all}\quad g\in\cO_X(X).$$
Finally, since $C_0$ is a proper closed subset of $W$, we have
$$A_X(w)\leq {\rm codim}(\overline{W})\leq {\rm codim}(\overline{C_0})-1=A_X(v)-1,$$
hence $w$ satisfies all the required conditions. Therefore it is enough to prove the claim.

Note that our assumptions imply that $m-e-2\geq e-1\geq 0$. In order to simplify the notation, we put $\psi=\psi_{m-e-2}\colon
X_{\infty}\to X_{m-e-2}$.
The key point is to describe, for every $\overline{\gamma}\in\psi(C')$, the cylinders
$\psi^{-1}(\overline{\gamma})\cap C_0\subseteq \psi^{-1}(\overline{\gamma})\cap C'$.
This is a local computation, based on Taylor's formula, that goes back to the proof of \cite[Lemma~3.4]{DenefLoeser}.
We choose an algebraic system of coordinates $x_1,\ldots,x_n$ in  a neighborhood of $P=\pi_{m-e-2}(\overline{\gamma})$,
centered at $P$. This gives an isomorphism $\widehat{\cO_{X,P}}\simeq\C\llbracket y_1,\ldots,y_n\rrbracket$ that maps each $x_i$ to $y_i$, and
let $\varphi\in \C\llbracket y_1,\ldots,y_n\rrbracket$ be the formal power series corresponding to $f$.
As we have seen, the system of coordinates also allows us to identify an arc on $X$ lying over $P$ with an element of $\big(t\C\llbracket t\rrbracket\big)^n$.

Let $\gamma\in C'$ be an arc with $\psi(\gamma)=\overline{\gamma}$, which corresponds to $u=(u_1,\ldots,u_n)\in \big(t\C\llbracket t\rrbracket\big)^n$,
so that $\gamma^*(f)=\varphi(u_1,\ldots,u_n)$. Any other arc $\delta\in \psi^{-1}(\overline{\gamma})$ corresponds to
$u+v$, for some $v=(v_1,\ldots,v_n)\in \big(t^{m-e-1}\C\llbracket t\rrbracket\big)^n$. It follows from the Taylor expansion of $\varphi$ that we have
\begin{equation}\label{eq_thm3_intro}
\delta^*(f)=\varphi(u+v)=\gamma^*(f)+\sum_{i=1}^n\gamma^*\left(\frac{\partial f}{\partial x_i}\right)v_i+\sum_{i,j=1}^n\gamma^*\left(\frac{\partial^2 f}{\partial x_i
\partial x_j}\right)v_iv_j+\text{higher order terms}.
\end{equation}
By assumption, we have ${\rm ord}_t\gamma^*(f)\geq m-1$ and
$$\min_{i=1}^n{\rm ord}_t\gamma^*\left(\frac{\partial f}{\partial x_i}\right)=e$$
(note that $m-1>e$). Since ${\rm ord}_tv_i\geq m-e-1$, an immediate computation using (\ref{eq_thm3_intro})
shows that for every such $\delta$, we have ${\rm ord}_t\delta^*(f)\geq m-1$. Since $m-e-1\geq e$, we also see that
${\rm ord}_t\delta^*(J_f)\geq e$, hence $\psi^{-1}(\overline{\gamma})\cap C'$ is the open subset of $\psi^{-1}(\overline{\gamma})$ defined by
having contact order with $J_f$ precisely $e$; in particular, this is an irreducible cylinder.

For every series $g\in\C\llbracket t\rrbracket$, let us write $g_0$ for the constant term of $g$.
We also write $v_{i,0}$ for the constant term of $v_i/t^{m-e-1}$.
It follows from (\ref{eq_thm3_intro})
that the coeficient of $t^{m-1}$ in $\delta^*(f)$ is equal to
$$\big(\gamma^*(f)/t^{m-1}\big)_0+\sum_{i=1}^n\big(\gamma^*(\partial f/\partial x_i)/t^e\big)_0v_{i,0}+
\sum_{i,j=1}^n\big(\gamma^*(\partial^2 f/\partial x_i\partial x_j)\big)_0v_{i,0}v_{j,0},$$
and the last term only appears if $m=2e+1$. In any case, since some $\big(\gamma^*(\partial f/\partial x_i)/t^e\big)_0$ is nonzero,
the vanishing of this coefficient defines a hypersurface in the affine space $\A^n$ parametrizing $(v_{1,0},\ldots,v_{n,0})$.
Since $C_0\subseteq {\rm Cont}^{\geq m}(f)$, it follows that $\psi^{-1}(\overline{\gamma})\cap C_0$ is different from $\psi^{-1}(\overline{\gamma})\cap C'$.
The fact that $C_0$ is not an irreducible component of $C'$ follows by taking the image in some $X_q$, with $q\gg 0$
(so that $C_0$ is the inverse image of a locally closed subset of $X_q)$, and applying the following easy lemma.
\end{proof}

\begin{lem}
Let $g\colon U\to V$ be a morphism of algebraic varieties and $W$  a closed subset of $U$
such that for every $y\in g(U)$, the following two conditions hold:
\begin{enumerate}
\item[i)] The fiber $g^{-1}(y)$ is irreducible, and
\item[ii)] The intersection $g^{-1}(y)\cap W$ is different from $g^{-1}(y)$.
\end{enumerate}
Then $W$ is not an irreducible component of $U$.
\end{lem}

\begin{proof}
Arguing by contradiction, suppose that the irreducible components of $U$ are
$U_1=W,U_2,\ldots,U_r$. If $x\in W\smallsetminus\bigcup_{i\geq 2}U_i$ and $y=g(x)$, then
$g^{-1}(y)\not\subseteq U_i$ for any $i\geq 2$. On the other hand,
$g^{-1}(y)$ is irreducible and contained in $\bigcup_{i\geq 1}\big(U_i\cap g^{-1}(y)\big)$,
hence $g^{-1}(y)\subseteq U_1$. This contradicts condition ii).
\end{proof}

We can now deduce the assertion about multiplier ideals:

\begin{proof}[Proof of Corollary~\ref{corC_intro}]
Since $(f)\subseteq (f)+J_f^2$, we clearly have the inclusion
$$\cJ(f^{\lambda})\subseteq\cJ\big((f,J_f^2)^{\lambda}\big)$$
for all $\lambda>0$. We now suppose that $\lambda<1$ and prove the reverse inclusion.

We may and will assume that $X$ is affine. Arguing by contradiction, suppose that we have $g\in \cJ\big((f,J_f^2)^{\lambda}\big)$
such that $g\not\in\cJ(f^{\lambda})$. By definition of multiplier ideals, the latter condition
implies  that there is a valuation $v={\rm ord}_E$,
where $E$ is a prime divisor on a log resolution of $(X,f)$ such that
\begin{equation}\label{eq_corC}
v(g)+A_X(v)\leq \lambda\cdot v(f).
\end{equation}
We choose such $E$ with the property that $A_X(v)$ is minimal.

If $v(J_f)\geq \frac{1}{2}v(f)$, then $v(f)=v(f,J_f^2)$, and
(\ref{eq_corC}) contradicts the fact that $g\in  \cJ\big((f,J_f^2)^{\lambda}\big)$.
Hence we may and will assume that $v(J_f)<\frac{1}{2}v(f)$. If $v(J_f)=0$, then there is an open subset $U$ of $X$
that intersects the center of $v$ on $X$ and such that $f\vert_U$ defines a smooth hypersurface. In this case, since $\lambda<1$,
we have $\cJ(f\vert_U^{\lambda})=\cO_U$. On the other hand, since $U$ intersects the center of $v$, it follows from (\ref{eq_corC})
that $g\vert_U\not\in \cJ(f\vert_U^{\lambda})$, a contradiction.

We thus may and will assume that
$$0<v(J_f)<\frac{1}{2}v(f).$$
In this case, it follows from Theorem~\ref{thm3_intro} that there is a divisorial valuation $w$ on $X$ that satisfies
properties i), ii), and iii) in the theorem. We thus have
$$A_X(w)\leq A_X(v)-1\leq \lambda\cdot v(f)-1-v(g)\leq \lambda\cdot \big(w(f)+1\big)-1-w(g),$$
where the first inequality follows from condition iii) and the third inequality follows from conditions i) and ii).
Since
$$\lambda\cdot \big(w(f)+1\big)-1-w(g)=\lambda\cdot w(f)-w(g)+\lambda-1\leq \lambda\cdot w(f)-w(g),$$
if we write $w=q\cdot {\rm ord}_F$, for some prime divisor $F$ over $X$, after dividing the above inequalities by $q$,
we obtain
$$A_X({\rm ord}_F)\leq \lambda\cdot {\rm ord}_F(f)-{\rm ord}_F(g)$$
and
$$A_X({\rm ord}_F)\leq A_X(w)\leq A_X(v)-1,$$
contradicting the minimality in the choice of $E$. The contradiction we obtained shows that we have in fact
$$\cJ\big((f,J_f^2)^{\lambda}\big)\subseteq \cJ(f^{\lambda}),$$
completing the proof of the first assertion in the corollary.

The proof of the second assertion is similar. We may and will assume that $X$ is affine. Recall that the
adjoint ideal ${\rm adj}(f)$ consists of all $g\in\cO_X(X)$ with the property that for every exceptional divisor $E$ over $X$,
we have
$${\rm ord}_E(g)>{\rm ord}_E(f)-A_X({\rm ord}_E)$$
(see \cite[Chapter~9.3.48]{Lazarsfeld}).
The inclusion
$$
{\rm adj}(f)\subseteq\cJ(f,J_f^2),
$$
is easy: if $g\in {\rm adj}(f)$, then for every divisorial valuation ${\rm ord}_E$ on $X$, we have
\begin{equation}\label{eq0_pf_cor}
{\rm ord}_E(g)>{\rm ord}_E(f,J_f^2)-A_X({\rm ord}_E).
\end{equation}
If $E$ is exceptional, this follows from the definition of the adjoint ideal and the fact that ${\rm ord}_E(f)\geq {\rm ord}_E(f,J_f^2)$.
On the other hand, if $E$ is a divisor on $X$, then $A_X({\rm ord}_E)=1$ and ${\rm ord}_E(f,J_f^2)=0$ (the latter equality follows from the fact that
if $f$ vanishes on $E$, since $f$ defines a reduced hypersurface, we have ${\rm ord}_E(J_f)=0$). Since (\ref{eq0_pf_cor})
holds for every $E$, we conclude that $g\in\J(f,J_f^2)$.

We now turn to the interesting inclusion
$$\cJ(f,J_f^2)\subseteq {\rm adj}(f).$$
Suppose that $g\in \cJ(f,J_f^2)$, but $g\not\in {\rm adj}(f)$. The latter condition implies that there is
an exceptional divisor $E$ over $X$ such that
\begin{equation}\label{eq_pf_cor}
{\rm ord}_E(g)\leq {\rm ord}_E(f)-A_X({\rm ord}_E).
\end{equation}
We choose such $E$ with $A_X({\rm ord}_E)$ minimal
and argue as in the proof of the first part.

If ${\rm ord}_E(J_f)\geq\frac{1}{2}{\rm ord}_E(f)$, then ${\rm ord}_E(f)={\rm ord}_E(f,J_f^2)$, and
(\ref{eq_pf_cor}) contradicts the fact that $g\in\cJ(f,J_f^2)$.
Hence we may and will assume that ${\rm ord}_E(J_f)<\frac{1}{2}{\rm ord}_E(f)$.
If ${\rm ord}_E(J_f)=0$, then there is an open subset $U$ of $X$ that intersects the center of ${\rm ord}_E$ on $X$
and such that $f\vert_U$ defines a smooth hypersurface. In particular,
we have ${\rm adj}(f)\vert_U=\cO_U$ by \cite[Proposition~9.3.48]{Lazarsfeld}, contradicting (\ref{eq_pf_cor}).

We thus may and will assume that
$$0<{\rm ord}_E(J_f)<{\rm ord}_E(f).$$
We then apply Theorem~\ref{thm3_intro} for $v={\rm ord}_E$
 to find a divisorial valuation $w=q\cdot {\rm ord}_F$ on $X$ that satisfies
properties i)-iv) in the theorem. We obtain
$$w(g)\leq {\rm ord}_E(g)\leq {\rm ord}_E(f)-A_X({\rm ord}_E)\leq w(f)-A_X(w).$$
Dividing by $q$, we obtain
$${\rm ord}_F(g)\leq {\rm ord}_F(f)-A_X({\rm ord}_F).$$
Since
$A_X({\rm ord}_F)\leq A_X(w)\leq A_X(v)-1$ and $F$ is an exceptional divisor
(we use here the fact that $c_X({\rm ord}_F)=c_X({\rm ord}_E)$), this contradicts the minimality in our choice of $E$.
This completes the proof of the corollary.
\end{proof}

Corollary~\ref{corC_intro} easily implies the assertions in Theorem~\ref{thmA_intro}.

\begin{proof}[Second proof of Theorem~\ref{thmA_intro}]
Recall that for a proper nonzero ideal ${\mathfrak a}$, we have
$$\lct({\mathfrak a})=\min\{\lambda>0\mid\cJ({\mathfrak a}^{\lambda})\neq\cO_X\}.$$
Since $\lct(f)\leq 1$, the first assertion in Corollary~\ref{corC_intro} implies that in general, we have
$$\lct(f)=\min\{\lct(f,J_f^2),1\}.$$
Note that if $\lct(f)=1$, then automatically $H$ is reduced. We thus deduce from the second
assertion in Corollary~\ref{corC_intro} that ${\rm lct}(f,J_f^2)>1$ if and only if $H$ is reduced
and ${\rm adj}(f)=\cO_X$. By \cite[9.3.48]{Lazarsfeld}, this holds if and only if $H$ has rational singularities.
\end{proof}

\section{Two examples}\label{examples}

In this sections we discuss two examples. We begin with the case of the generic determinantal hypersurface,
for which we show that the inequality in Theorem~\ref{thmB_intro} is an equality.

\begin{eg}\label{eg_det}
Let $n\geq 2$ and $X=\A^{n^2}$ be the affine space of $n\times n$ matrices, with coordinates $x_{i,j}$, for $1\leq i,j\leq n$.
We consider $f={\rm det}(A)$, where $A$ is the matrix $(x_{i,j})_{1\leq i,j\leq n}$. In this case,
 it is well-known that the Bernstein-Sato polynomial of $f$ is given by
$$b_f(s)=\prod_{i=1}^n(s+i)$$
(see, for example, \cite[Appendix]{Kimura}). We thus have $\widetilde{\alpha}_f=2$.

In order to compute $\lct(f,J_f^2)$, we use the arc-theoretic description reviewed in the previous section, together with
the approach in the determinantal case due to Docampo \cite{Docampo}. Note first that $J_f$ is the ideal of $\cO_X(X)$
generated by the $(n-1)$-minors of the matrix $A$. We use the action of $G={\rm GL}_n(\C)\times {\rm GL}_n(\C)$ on
$X$ given by $(g,h)\cdot A=gAh^{-1}$. We have an induced action of $G_{\infty}={\rm GL}_n(\C\llbracket t\rrbracket)\times
{\rm GL}_n(\C\llbracket t\rrbracket)$
on $X_{\infty}$ and we consider the orbits with respect to this action. Since the ideal $(f)+J_f^2$ is preserved by the $G$-action, it follows that
every contact locus of this ideal is a union of $G_{\infty}$-orbits.

Recall that a $G_{\infty}$-orbit in $X_{\infty}$ of finite codimension corresponds to a sequence of integers $\lambda_1\geq\ldots,\geq\lambda_n\geq 0$,
such that the orbit consists of those $n\times n$ matrices with entries in $\C\llbracket t\rrbracket$ that are equivalent via Gaussian elimination to a diagonal matrix having on the diagonal $t^{\lambda_1},\ldots,t^{\lambda_n}$. The codimension of this orbit is $\sum_{i=1}^n\lambda_i(2i-1)$ by
\cite[Theorem~C]{Docampo}. Moreover, it is clear that the order of an arc in this orbit along the ideal $(f)+J_f^2$ is
$$\min\left\{\sum_{i=1}^n\lambda_i, 2\cdot\sum_{i=2}^n\lambda_i\right\}.$$
The description of the log canonical threshold in terms of contact loci (see \cite{ELM}) thus implies
\begin{equation}\label{eq_eg1}
\lct(f,J_f^2)=\min_{\lambda}\frac{\sum_{i=1}^n\lambda_i(2i-1)}{\min\left\{\sum_{i=1}^n\lambda_i, 2\cdot\sum_{i=2}^n\lambda_i\right\}},
\end{equation}
where the minimum is over all $\lambda=(\lambda_1,\ldots,\lambda_n)$, with $\lambda_1\geq\ldots\geq\lambda_n\geq 0$
and $\lambda_2>0$. If we take $\lambda_1=\lambda_2>0$ and $\lambda_3=\ldots=\lambda_n=0$, then the expression on the right-hand side of
(\ref{eq_eg1}) is $2$. In order to see that for all $\lambda$ this expression is $\geq 2$,  note that since $\lambda_1\geq\lambda_2$, we have
$$\sum_{i=1}^n(2i-1)\lambda_i\geq 4\cdot\sum_{i=2}^n\lambda_i.$$
We thus conclude that $\lct(f,J_f^2)=2=\widetilde{\alpha}_f$.
\end{eg}

We next turn to the case of homogeneous diagonal hypersurfaces, where we will see that the inequality
in Theorem~\ref{thmB_intro} can be strict.

\begin{eg}\label{eg_diag}
Let $X=\A^n$, with $n\geq 2$, and $f=x_1^d+\ldots+x_n^d$, for some $d\geq 2$. In this case, it is known that the Berstein-Sato polynomial of $f$ is given by
$$b_f(s)=(s+1)\cdot\prod_{1\leq b_1,\ldots,b_n\leq d-1}\left(s+\sum_{i=1}^n\frac{b_i}{d}\right)$$
(see \cite[Proposition~3.6]{Yano}). In particular, we have $\widetilde{\alpha}_f=\frac{n}{d}$.

We will show that
\begin{equation}\label{eq_eg_diag}
\lct(f,J_f^2)=\min\left\{\frac{n+d-2}{2d-2},\frac{n}{d}\right\}.
\end{equation}
It is clear that we have $J_f=(x_1^{d-1},\ldots,x_n^{d-1})$. Let $\pi\colon Y\to X$ be the blow-up of $X$
at the origin. By symmetry, it is enough to consider the chart $U$
on $Y$ with coordinates $y_1,\ldots,y_n$ such that $x_1=y_1$ and $x_i=y_1y_i$ for $i\geq 2$.
In this chart we have $f\circ\pi\vert_U=y_1^dg$, with $g=1+y_2^d+\ldots+y_n^d$, and $J_f^2\cdot\cO_U=(y_1^{2d-2})$.
We deduce that a log resolution of the ideal $\big((f)+J_f^2\big)\cdot\cO_U$ can be obtained by blowing-up $U$ along the ideal $(g,y_1^{d-2})$ and then resolving
torically the resulting variety. In particular, we see that in order to compute $\lct(f,J_f^2)$ it is enough to consider toric divisors over $Y$,
with respect to the system of coordinates given by $y_1$ and $g$. If $E$ is such a divisor, with ${\rm ord}_E(g)=a$ and ${\rm ord}_E(y_1)=b$,
with $a$ and $b$ nonnegative integers, not both $0$, we have
$${\rm ord}_E\big((f)+J_f^2\big)=\min\{db+a, (2d-2)b\}\quad\text{and}\quad A_X({\rm ord}_E)=nb+a.$$
This implies that
$$\lct(f,J_f^2)=\min_{(a,b)}\frac{nb+a}{\min\{db+a,(2d-2)b\}},$$
where the minimum is over all $(a,b)\in\Z_{\geq 0}^2\smallsetminus\{(0,0)\}$.
It is now a straightforward exercise to deduce the formula (\ref{eq_eg_diag}).

If $d=2$, then we see that $\lct(f,J_f^2)=\frac{n}{2}=\widetilde{\alpha}_f$ (of course, this can be seen directly, since in this case
$(f)+J_f^2=(x_1,\ldots,x_n)^2$). Suppose now that $d\geq 3$. In this case, a straightforward calculation shows that
$$\frac{n+d-2}{2d-2}<\frac{n}{d}\quad\text{if and only if}\quad d<n.$$
 We thus see that if $d\geq n$ (which is precisely the case when the hypersurface defined by $f$ does not have rational
singularities), then $\lct(f,J_f^2)=\frac{n}{d}=\lct(f)$, as expected according to Theorem~\ref{thmA_intro}. On the other hand, if $3\leq d<n$, then
we have a strict inequality $\lct(f,J_f^2)<\widetilde{\alpha}_f$.
\end{eg}

\section{The motivic oscillation index and $\lct(f,J_f^2)$}\label{moi}

We begin by recalling the definition of the motivic oscillation index from \cite{CMN}.
Let $\Qalg$  be the algebraic closure of $\Q$ inside $\C$.
We consider a nonconstant polynomial $f$ in $\Qalg[x_1,\ldots,x_n]$ and a closed subscheme  $Z$ of $\A_{\Qalg}^n$.

 Let $K$ be a number field  such that $f$ and $Z$ are defined over $K$. Choose an integer $N>0$ such that $f$ and a set of generators for the ideal
 defining $Z$ lie in $\cO[1/N]$, where $\cO$ is the ring of integers of $K$. For any prime $p$ not dividing $N$, any completion $L$  of $K$ above $p$,
and any nontrivial additive character $\psi\colon L\to \C^\times$, consider the following integral
\begin{equation}\label{eq:EfZ}
E_{f,L,\psi}^{Z} := \int_{\{x\in \cO_{L}^n\mid \overline x\in Z(k_L) \} }  \psi\big(f(x)\big)|dx| ,
\end{equation}
where $\cO_L$ stands for the valuation ring of $L$ with residue field $k_L$,
$\overline x$ stands for the image of $x$ under the natural projection $\cO_{L}^n \to k_L^n$, and $|dx|$ is the Haar measure on $L^n$ normalized so that
$\cO_L^n$ has measure $1$. Writing $q_L$ for the number of elements of $k_L$, let $\sigma_L$ be the supremum over all $\sigma\geq 0$ such that
$$
|E_{f,L,\psi}^{Z}| \leq c q_L^{-m\sigma}
$$
for some $c=c(\sigma,f,L)$ which is independent of $\psi$, and where $m$ is such that $\psi$ is trivial on $\cM_L^m$ and nontrivial on $\cM_L^{m-1}$,
with $\cM_L$ being the maximal ideal of $\cO_L$.  Note that $\sigma_L$ can equal $+\infty$; this is the case precisely when
the morphism $\A_{\Qalg}^n\to\A_{\Qalg}^1$ defined by $f$ is smooth in an open neighborhood of $Z$.
We define the $K$-oscillation index of $f$ along $Z$ as
$$
\oi_Z(f) := \lim_{M\to\infty}\inf_{L} \sigma_L,
$$
where the infimum is taken over all non-Archimedean completions $L$ of $K$ above primes $p_L$ with $p_L>M$. 
Finally, the motivic oscillation index of $f$ along $Z$ is defined as
$$
\moi_Z(f) := \inf_{K} \oi_Z(f),
$$
where $K$ runs over all number fields satisfying the above conditions.
This definition corresponds to the definition given in \cite[Section~3.4]{CMN} by Igusa's work (see \cite{Denef}), which relates upper bounds for oscillating integrals with nontrivial poles of local zeta functions.
Note that the variant $\moi(f)$ we considered in the Introduction corresponds to the case when $Z$ is the hypersurface defined by $f$ (note the small change in notation for $\moi(f)$, compared to \cite{CMN}).

The following is the main result of this section.

\begin{thm}\label{thm_moi2}
If $f\in\Qalg[x_1,\ldots,x_n]$ is a nonconstant polynomial and $Z$ is a closed subscheme of the hypersurface defined by $f$, then
\begin{equation}\label{eq:moifJ2}
\moi_Z(f) \geq \lct_{Z}(f,J_f^2).
\end{equation}
In addition, if $\lct_{Z}(f,J_f^2)\leq 1$, then equality holds in (\ref{eq:moifJ2});
in this case, we also have $\oi_Z(f)  = \lct_{Z}(f,J_f^2)$
for every number field $K$ such that $f$ and $Z$ are defined over $K$.
\end{thm}

We recall that for every nonzero ideal ${\mathfrak a}$ in $\Qalg[x_1,\ldots,x_n]$, by definition
$$\lct_Z({\mathfrak a})=\max_U\lct({\mathfrak a}\vert_U),$$
where the maximum is over all open neighborhoods $U$ of $Z$. We also note that if ${\mathfrak a}_{\C}$
is the extension of ${\mathfrak a}$ to $\C[x_1,\ldots,x_n]$, then $\lct({\mathfrak a})=\lct({\mathfrak a}_{\C})$.
We will derive Theorem~\ref{thm_moi2} from the following proposition and the definition of the oscillation index.

\begin{prop}\label{prop:E^ZfJf}
Let $K$ be any number field such that $f$ and $Z$ are defined over $K$ and let
$L$ be a non-Archimedean completion of $K$ above a prime $p_L$, with residue field $k_L$ with $q_L$ elements.
If $p_L$ is large enough (in terms of the data $f,Z,K$), then for every $\varepsilon>0$, there exists a constant $c = c(f,Z,L,\varepsilon)$ such that for each nontrivial additive character $\psi\colon L\to\C^\times$ we have
$$
|E^{Z}_{f,L,\psi}| <  c q_L^{- m\sigma }\quad\text{for}\quad
\sigma =  \lct_Z(f,J_f^2) - \varepsilon,
$$
where $m=m(\psi)$ is such that $\psi$ is trivial on $\cM_L^m$ and nontrivial on $\cM_L^{m-1}$.
\end{prop}

\begin{proof}[Proof of Proposition \ref{prop:E^ZfJf}]
By Igusa's results from \cite{Igusa} and \cite{Denef}, recalled in \cite[Propositions~3.1 and 3.4]{CMN}, we have for each $\psi$ with $m=m(\psi)>1$
\begin{equation}\label{eq:EfZ}
E_{f,L,\psi}^{Z} = \int_{\{x\in \cO_{L}^n\mid \overline x\in Z(k_L),\ \ord f(x)\geq m-1 \} }  \psi\big(f(x)\big)|dx| ,
\end{equation}
where $\overline x$ stands for the image of $x$ under the natural projection $\cO_{L}^n \to k_L^n$, and where $p_L$ is assumed to be large. We now show
that we also have
\begin{equation}\label{eq:EfZJ}
E_{f,L,\psi}^{Z}= \int_{ \{x\in \cO_{L}^n\mid \overline x\in Z(k_L),\ \ord(f(x),J_f^2(x))\geq m-1 \} }  \psi\big(f(x)\big)|dx|,
\end{equation}
by adapting the proof of \cite[Proposition~2.1]{CHerr}. Here, the condition $\ord\big((f(x),J_f^2(x)\big)\geq m-1$ for $x\in \cO_L^n$ means that we have $\ord\big(g(x)\big)\geq m-1$ for every polynomial $g$ in the ideal of the polynomial ring over $\cO_L$ generated by $f$ and $J_f^2$.  We
 use the orthogonality of characters in the form
\begin{equation}\label{eq:orthchar}
\int_{z \in \cM_L^{m-1}} \psi(z)|dz| = 0
\end{equation}
for $m=m(\psi)$ and combine this with the Taylor expansion, as follows. Let $x_0$ be a point in $\cO_L^m$ such that $\ord f(x_0)\geq m-1$ and suppose that
$\ord\big(J_f^2(x_0)\big) < m-1$. In order to prove (\ref{eq:EfZJ}), it is enough to show that for every such $x_0$, we have
\begin{equation}\label{eq:orth}
\int_{x\in x_0+(\cM_L^{\overline m})^n} \psi\big(f(x)\big) |dx| = 0,
\end{equation}
with $\overline m$ equal to $m/2$ if $m$ is even, and equal to $(m+1)/2$ if $m$ is odd.
If $x=y+x_0$, with $y=(y_1,\ldots,y_n)\in (\cM_L^{\overline m})^n$, then we write the Taylor expansion of $f$ around $x_0$:
$$
f(x) =  f(y+x_0) =   a_0 + \sum_{i=1}^n a_i y_i + \mbox{higher order terms in } y.
$$
Since $p_L$ is assumed to be large and $y\in (\cM_L^{\overline m})^n$, we see  that
$$
f(x) \equiv a_0 + \sum_{i=1}^n a_i y_i \quad (\bmod \cM_L^m)
$$
and thus, since $\psi$ is trivial on $\cM_L^m$, we obtain
\begin{equation}\label{eq:linear}
\int_{x\in x_0+(\cM_L^{\overline m})^n} \psi\big(f(x)\big) |dx| = \int_{y\in (\cM_L^{\overline m})^n} \psi \left(a_0 + \sum_{i=1}^n a_i y_i\right) |dy|.
\end{equation}
Note that the condition $\ord\big(J_f^2(x_0)\big) <m-1$ implies that $\min_i \ord a_i < (m-1)/2$. Since $(m-1)/2 + {\overline m} \leq m$, using the orthogonality relation (\ref{eq:orthchar}), we deduce (\ref{eq:orth}) from (\ref{eq:linear}). This completes the proof of (\ref{eq:EfZJ}).

By (\ref{eq:EfZJ}), using the fact that $| \psi(x) |=1$ for all $x$ in $L$, we get
$$
|E_{f,L,\psi}^{Z} | \leq \Vol\big( \{ x\in \cO_{L}^n\mid \overline x \in Z(k_L),\ \ord(f(x),J_f^2(x))\geq m-1 \}\big)  ,
$$
where the volume is taken with respect to the Haar measure $|dx|$ on $L^n$.
Therefore the existence of $c$ as desired follows from Corollary 2.9 of \cite{VeysZ} and the two sentences following that corollary, which give a link between the log canonical threshold at $Z$ of any ideal ${\mathfrak a}$ of $\cO_L[x]$ and the volume of
 $\{ x\in \cO_{L}^n\mid \overline x\in Z(k_L),\ \ord\big({\mathfrak a}(x)\big)\geq m \}$ uniformly in
 $m>1$. 
\end{proof}

\begin{proof}[Proof of Theorem \ref{thm_moi2}]
The formula (\ref{eq:moifJ2}) follows from Proposition \ref{prop:E^ZfJf}
and the definition of $\moi_Z(f)$. Moreover, if $\lct_{Z}(f,J_f^2)\leq 1$, then
part ii) of Theorem~\ref{thmA_intro} 
gives that the hypersurface defined by $f$ does not have
rational singularities in any open neighborhood of $Z$.
Under this last condition, it follows from 
\cite[Proposition~3.10]{CMN} that $\moi_Z(f) =\lct_Z(f)$, which together with (\ref{eq:moifJ2})  and $\lct_Z(f) \leq \lct_{Z}(f,J_f^2)$ implies $\moi_Z(f)  = \lct_{Z}(f,J_f^2) = \lct_Z(f)$.   The fact that
when $\lct_{Z}(f,J_f^2)\leq 1$ we also have
$\oi_Z(f) =\lct_Z(f)$  follows by slightly adapting the proof of \cite[Proposition~3.10]{CMN}. 
\end{proof}

\end{document}